\documentclass[a4paper,twoside,reqno]{amsart}
\usepackage{amssymb}
\usepackage{amsmath}
\usepackage[mathscr]{eucal}
\usepackage{hyperref}
\usepackage{graphicx}
\usepackage[all]{xy}
\usepackage{tikz}
\usepackage{tikz-cd} 
\usepackage{multirow}
\usepackage{booktabs}
\usepackage{float}
\usepackage{placeins}
\usepackage[capitalise]{cleveref}
\usetikzlibrary{shapes}

\crefformat{equation}{equation~(#1)}

\theoremstyle{plain}
\newtheorem{prop}{Proposition}[section]
\newtheorem{thm}[prop]{Theorem}
\newtheorem{cor}[prop]{Corollary}
\newtheorem{lem}[prop]{Lemma}

\newtheorem*{thm*}{Theorem}

\theoremstyle{definition}
\newtheorem{dfn}[prop]{Definition}
\newtheorem{rem}[prop]{Remark}
\newtheorem{example}[prop]{Example}
\newtheorem{lab}[prop]{}
\newtheorem{notation}[prop]{Notation}

\newcommand{\C}{{\mathbb{C}}}
\renewcommand{\P}{{\mathbb{P}}}
\newcommand{\R}{{\mathbb{R}}}
\newcommand{\N}{{\mathbb{N}}}
\newcommand{\Z}{{\mathbb{Z}}}
\newcommand{\bbS}{{\mathbb{S}}}

\DeclareMathOperator{\im}{im}

\DeclareMathOperator{\interior}{int}

\DeclareMathOperator{\rk}{rk}

\DeclareMathOperator{\gl}{GL}
\DeclareMathOperator{\pgl}{PGL}

\DeclareMathOperator{\gram}{Gram}

\DeclareMathOperator{\Aut}{Aut}

\DeclareMathOperator{\spn}{span}
\DeclareMathOperator{\codim}{codim}

\DeclareMathOperator{\orth}{O} 

\newcommand{\Rx}{\mathbb{R}[\ul{x}]}
\newcommand{\Cx}{\mathbb{C}[\ul{x}]}
\newcommand{\vc}{\mathcal{V}}

\newcommand{\du}{{\scriptscriptstyle\vee}}
\renewcommand{\setminus}{\smallsetminus}
\newcommand{\ol}{\overline}
\newcommand{\ul}{\underline}

\renewcommand{\subset}{\subseteq}

\newcommand{\mS}{\mathcal{S}}
\newcommand{\mT}{\mathcal{T}}

\newcommand{\sy}[1]{\mathcal{S}_2 #1}

\newcommand{\sa}{semi-alge\-braic}

\newcommand{\bp}{base-point}
\newcommand{\gs}{Gram spectrahedron}
\newcommand{\gsa}{Gram spectrahedra}
\newcommand{\cc}{change of coordinates}

\newcommand{\ldr}{linear symmetric determinantal representation}
\newcommand{\qdr}{quadratic symmetric determinantal representation}

\renewcommand{\emptyset}{\varnothing}
\renewcommand{\setminus}{\smallsetminus}
\renewcommand{\epsilon}{\varepsilon}
\renewcommand{\theta}{\vartheta}

\newcommand{\todfn}[1]{\textit{#1}}

\author[Julian Vill]{Julian Vill}
\address{	
	University of Konstanz, Germany, Fachbereich Mathematik und Statistik,
	D-78457 Konstanz, Germany
}
\email{julian.vill@uni-konstanz.de}

\title{Gram Spectrahedra of Ternary Quartics}
\date{\today}

\begin{document}

\begin{abstract}
The Gram spectrahedron of a real form $f\in\R[\ul x]_{2d}$ parametrizes all sum of squares representations of $f$. It is a compact, convex, \sa\ set, and we study its facial structure in the case of ternary quartics, i.e. $f\in\R[x,y,z]_4$. 
We show that the \gs\ of every smooth ternary quartic has faces of dimension 2, and generically none of dimension 1, thus answering a question in \cite{psv2011} about the existence of positive dimensional faces on such \gsa. We complete the proof in \cite{psv2011} showing that the so called Steiner graph of every smooth quartic is isomorphic to $K_4\coprod K_4$.
Moreover, we show that the \gs\ of a generic psd ternary quartic contains points of all ranks in the Pataki interval.
\end{abstract}

\maketitle

\section{Introduction}

Given a real homogeneous polynomial $f\in\R[x_1,\dots,x_n]_{2d}$ of degree $2d$ ($d\in\N$) that is a sum of squares, i.e. that can be written in the form $f=\sum_{i=1}^r p_i^2$ for some $p_1,\dots,p_r\in\R[x_1,\dots,x_n]_d$. In general, there are many non-equivalent ways to write $f$ in such a form. The \gs\ of $f$ parametrizes all sos-representations up to orthogonal equivalence. It carries in a natural way the structure of a spectrahedron, especially, it is a closed, convex and \sa\ set. As such its boundary is the union of its faces.

The facial structure of spectrahedra has first been studied by Ramana and Goldman \cite{rg1995}. In \cite{scheiderer2018} Scheiderer formulates these results in a coordinate-free way which we introduce in \cref{sec:facial structure of gram spectrahedra}. To this end we define the \gs\ of $f$ as the set
\[
\gram(f)=\{\theta\in\sy{\R[x_1,\dots,x_n]_d}\colon \theta\succeq 0, \mu(\theta)=f\}
\]
consisting of all symmetric, positive semidefinite tensors $\theta$ that are mapped to $f$ via the Gram map $\mu\colon\sy{\R[x_1,\dots,x_n]_d}\to\R[x_1,\dots,x_n]_{2d},\quad \sum_{i=1}^r p_i\otimes q_i\mapsto \sum_{i=1}^r p_iq_i$.

We are mostly concerned with the case $n=3,\, d=2$ of ternary quartics. 
Sum of squares representations of ternary quartics have already been studied by Hilbert \cite{hilbert1888} in 1888, who showed that every real psd ternary quartic can also be written as a sum of squares, and every such quartic can be written as a sum of three squares. Later, it was shown in \cite{prss2004} that every smooth psd quartic admits exactly eight such representations as a sum of three squares (up to orthogonal equivalence).

In \cite{psv2011} the authors relate the sum of squares representation of length 3 to the 28 bitangent lines of the quartic curve and their combinatorial structure. A modern treatise of the combinatorics of bitangents can be found in Dolgachev's book \cite{dolgachev2012} which we regularly use as a reference. They also observe that the eight corresponding Gram tensors of rank 3 split into two groups of four such that in every group all line segments between two of them are contained in the boundary of $\gram(f)$. The stronger statement that additionally for any two Gram tensors in different groups the line segment is not part of the boundary is also claimed for generic forms but only partially proven.
We complete the proof of this statement for all smooth quartics. The graph having these eight Gram tensors as vertices is called the Steiner graph of the form $f$.

In this paper we largely complete the picture of the facial structure for \gsa\ of ternary quartics. First, we proof that the Steiner graph of any smooth psd ternary quartic is isomorphic to $K_4\coprod K_4$ (\cref{thm:steiner_graph_non_deg}). This mostly finishes our study of the rank 3 Gram tensors and we turn to tensors of rank 4 and 5. Up to now, these were not understood at all, it was even unknown whether there exist any faces of positive dimension other than the faces containing the line segments represented in the Steiner graph (\cite[p.18]{psv2011}). These rank 4 and rank 5 tensors form the whole boundary of $\gram(f)$ together with the eight rank 3 tensors for any smooth form $f$.
We show that faces of rank 4 may either be extreme points or have dimension 1. These 1-dimensional faces however, may only appear on \gsa\ of smooth forms $f$ if the automorphism group of the curve $\{f=0\}$ has even order (\cref{cor:one_dim_face_automorphism}). Faces of rank 5 are very different in  this aspect. Firstly, faces of rank 5 can only be of dimension 0 or 2 if the quartic is smooth (\cref{cor:possiblefaces}). Moreover, the Gram spectrahedron of a generic psd quartic has faces of rank 5 both of dimension 0 and 2 (\cref{cor:six_one_dim_faces}).
As the \gs\ of every smooth psd quartic has extreme points of rank 4 (\cref{prop:smooth_rk4_tensors}) this also shows that for generic psd quartics, the \gs\ has points of all ranks in the Pataki interval which contains the numbers $3,4$ and $5$. This is the smallest combination of dimension and size of the matrices where the Pataki interval has length three.

We briefly comment on the methods and the organization of the paper. There are two very different types of arguments used in the paper. To understand dimensions of faces, we make use of the fact that these can be calculated purely algebraically. Let $f\in\Sigma_{n,2d}$ and $F\subset\gram(f)$ a face. Let $\theta\in F$ be a relative interior point with image $U\subset\R[x_1,\dots,x_n]_d$. The dimension of $F$ is now given by $\binom{\dim U+1}{2}-\dim U^2$ where $U^2$ is the subspace of $\R[x_1,\dots,x_n]_{2d}$ spanned by all products $pq$ with $p,q\in U$. It is therefore enough to study possible dimensions of squares of subspaces.
On the other hand, if we want to show the existence of faces with certain properties we use the well-studied, classical approach of connecting sos-representations of a smooth form $f$ with the 28 bitangents of the curve $\{f=0\}$. This is especially necessary to show additional properties of the Steiner graph (e.g. \cref{prop:six_one_dim_faces}).

As to the organization of the paper. In \cref{sec:facial structure of gram spectrahedra} we review results about the facial structure of \gsa\ as shown in \cite{rg1995} and \cite{scheiderer2018}. \cref{sec:introduction to ternary quartics} contains a brief introduction to the classical theory of bitangents and determinantal representations of ternary quartics.
\cref{sec:Steiner graph} is devoted to the study of the Steiner graph. We start by recalling some facts about the configuration of bitangents, mostly using Dolgachev's book \cite{dolgachev2012} as a reference. We then show that the Steiner graph of any smooth psd ternary quartic is isomorphic to $K_4\coprod K_4$ (\cref{thm:steiner_graph_non_deg}), thereby also recalling the part of the proof contained in \cite{psv2011}.
In \cref{sec:the gram spectrahedron} we study faces of rank 4 and rank 5 on \gsa. If the psd ternary quartic $f$ is smooth, faces of rank 5 may only have dimension 0 or 2 (\cref{cor:possiblefaces}) and there always exists a face of dimension 2 (\cref{cor:six_one_dim_faces}). On the other hand, faces of rank 4 are always extreme points if the automorphism group of the curve $\{f=0\}$ does not have even order which is the case for generic $f$. In the case where the automorphism group does have even order, there may be faces of rank 4 and dimension 1 and we also relate this to the configuration of the bitangents (\cref{thm:1dimface}).

\section{Facial structure of Gram spectrahedra}
\label{sec:facial structure of gram spectrahedra}

We first explain the facial structure of \gsa. All results are originally due to Ramana and Goldman \cite{rg1995}, and the coordinate-free approach we are going to use has been introduced by Scheiderer \cite{scheiderer2018}.

Let $n,d\ge 1$. For any subspace $V\subset\Rx_d$, $\ul x=(x_1,\dots,x_n)$, we denote by $\sy{V}$ the second symmetric power of $V$ and identify it with the subset of $V\otimes V$ consisting of all symmetric tensors.
A symmetric tensor $\theta=\sum_{i=1}^r p_i\otimes q_i\in\sy{V}$ induces a symmetric bilinear map $b_\theta\colon V^\du\times V^\du\to\R,\, (\lambda,\nu)\mapsto\sum_{i=1}^r \lambda(p_i)\nu(q_i)$.
Furthermore, this also induces a linear map $\phi_\theta\colon V^\du\to V$ via $\lambda\mapsto b_\theta(\cdot,\lambda)=b_\theta(\lambda,\cdot)$.

We define the \todfn{image} of $\theta$ as the image of the linear map $\phi_\theta$, written as $\im\theta$, and the \todfn{rank} of $\theta$ as the dimension of its image, denoted by $\rk\theta$. Especially, this means that if $p_1,\dots,p_r$ and $q_1,\dots,q_r$ are linearly independent sets, then $\im\theta=\spn(\ul p)=\spn(\ul q)$.

We say that the tensor $\theta\in\sy{V}$ is \todfn{positive semidefinite}, $\theta\succeq 0$ , if $b_\theta(\lambda,\lambda)\ge 0$ for every $\lambda\in V^\du$ and write $\sy^+{V}$ for the set of all psd tensors in $\sy{V}$. Additionally, if $b_\theta(\lambda,\lambda)>0$ for every $0\neq \lambda\in V^\du$ we say that $\theta$ is \todfn{positive definite} and write $\theta\succ 0$.

After choosing a basis $p_1,\dots,p_r$ of $V$, every tensor $\theta\in\sy{V}$ can be written as $\theta=\sum_{i,j=1}^r a_{ij} p_i\otimes p_j$ with $a_{ij}=a_{ji}\in\R$. 
This identifies $\theta$ with the symmetric matrix $(a_{ij})_{i,j}\in\bbS^r$. Hence, $\sy{V}$ is identified with the space of all symmetric $r\times r$ matrices and $\sy^+{V}$ with the cone of all psd $r\times r$ matrices. Especially, $\sy^+{V}$ is a full-dimensional, convex cone inside $\sy{V}$.

Furthermore, since every real symmetric matrix can be diagonalized, we can write $\theta=\sum_{i=1}^r a_i q_i\otimes q_i$ for some $q_1,\dots,q_r\in V$ and $a_1,\dots,a_r\in\R$. Moreover, $\theta$ is psd if and only if all $a_i$ are non-negative.

The multiplication map $\Rx\otimes \Rx\to \Rx,\, f\otimes g\mapsto fg$ induces a multiplication map $\mu\colon \sy{\Rx_d}\to \Rx_{2d}$ on the symmetric tensors. This is exactly the Gram map, written in a coordinate-free way.
For any $f\in\Sigma_{n,2d}$, we therefore define 
\[
\gram(f):=\mu^{-1}(f)\cap\sy^+{\Rx_d}.
\]
We denote by $V^2$ the subspace of $\Rx_{2d}$ spanned by all product $pq$ with $p,q\in V$.

Let $f\in\Sigma_{n,2d}$ and let $\theta\in\gram(f)$. By diagonalizing, the Gram tensor $\theta$ gives rise to a sos-representation $f=\sum_{i=1}^r p_i^2$. If $p_1,\dots,p_r$ are linearly independent, the image of $\theta$ is given by the span of the forms $p_1,\dots,p_r$, and $r$ is equal to the rank of $\theta$, or equivalently the dimension of its image.

\begin{thm}[{\cite[Proposition 3.6.]{scheiderer2018}}]
\label{thm:dim_faces}
Let $f\in \Sigma_{n,2d}$. Let $F\subset\gram(f)$ be a face, and let $\theta\in F$ be a relative interior point of $F$. With $U:=\im(\theta)$, the dimension of $F$ is given by
\[
\dim F= \dim \sy{U}-\dim U^2=\binom{\dim U+1}{2}-\dim U^2.
\]

Moreover, if $W\subset V$ is a subspace such that $f=\sum_{i=1}^r p_i$ for some basis $p_1,\dots,p_r$ of $W$, then $\gram(f)$ has a face $G$ such that every relative interior point $\theta$ of $G$ satisfies $\im(\theta)=W$. The dimension of $G$ is given by $\sy{W}-\dim W^2$.
\end{thm}

We also call complex tensors Gram tensors and only restrict to real tensors whenever necessary. I.e.,
let $f\in \Rx_{2d}$, then any $\theta\in\sy{\Cx_d}$ is called a \todfn{Gram tensor} of $f$ if $\mu(\theta)=f$ if we extend $\mu$ by tensoring everything with $\C$.

\begin{example}
Consider the Fermat quartic $f=x^4+y^4+z^4\in\R[x,y,z]_4$, then we have the rank 3 sos decomposition $f=(x^2)^2+(y^2)^2+(z^2)^2$. The corresponding Gram tensor can for example be written as $\theta=x^2\otimes x^2+y^2\otimes y^2+z^2\otimes z^2$, since we have $\mu(\theta)=(x^2)^2+(y^2)^2+(z^2)^2=f$ and $\theta$ is psd.

As with Gram matrices, any diagonalized Gram tensor of a form $f$ immediately gives rise to a sos representation of $f$.
\end{example}

\begin{example}
Consider again the Fermat quartic $f=x^4+y^4+z^4$. Since $f$ is positive definite, we see that $f\in\interior\Sigma_{3,4}$. Then the dimension of the \gs\ of $f$ is given by 
\[
\dim\gram(f)=\binom{\dim \Rx_2 +1}{2}-\dim \Rx_4=21-15=6.
\]
Moreover, the Gram tensor $\theta=x^2\otimes x^2+y^2\otimes y^2+z^2\otimes z^2$ is an extreme point of $\gram(f)$. Indeed, the six forms $x^4,x^2y^2,x^2z^2,y^4,y^2z^2,z^4$ are linearly independent and therefore $\theta$ is an extreme point of $\gram(f)$ by \cref{thm:dim_faces}.
\end{example}

\begin{rem}
If we are working with a fixed number of variables and a fixed degree, the dimension of a face only depends on the dimension of $U$ and the dimension of $U^2$. This means that we can determine dimensions purely algebraically.
It turns out that it is more convenient to talk about the codimensions of $U$ and $U^2$ some of the time. In these cases this should always be understood as the codimension of $U$ as a subspace of $\Rx_d$. Analogously reading $U^2$ as a subspace of $\Rx_{2d}$.
\end{rem}

A special kind of face is the extreme point. These are by definition faces of dimension $0$. The next result is originally due to Pataki in \cite{pataki2000}, this formulation can be found in \cite[Proposition 3.1.]{cpsv2017}. It determines bounds for the ranks of extreme points of generic spectrahedra and \gsa.

\begin{prop}
\label{prop:pataki_range}
Let $\dim V=n$, let $L\subset \sy{V}$ be an affine-linear subspace with $\dim L=m$, and let $S=L\cap\sy^+{V}$.
\begin{enumerate}
\item For every extreme point $\theta$ of $S$, the rank $\rk\theta=r$ satisfies
\[
m+\binom{r+1}{2}\le\binom{n+1}{2}.
\]
\item When $L$ is chosen generically among all affine-linear subspaces of dimension $m$, every $\theta\in S$ satisfies
\[
m\ge \binom{n-\rk\theta+1}{2}.
\]
\end{enumerate}
\end{prop}

\begin{example}
Let $n=3,\,d=2$, and let $f\in\Sigma_{3,4}$ be a generic ternary quartic. Then the ranks $r$ of extreme points of the \gs\ $\gram(f)$, satisfy the inequalities
\[
\binom{r+1}{2}\le \dim \Rx_4=15 \text{ and } \binom{7}{2} \ge \binom{7-r}{2}+15
\]
which is equivalent to $r\le 5$ and $r\ge 3$. Hence, the Pataki interval is given by $3\le r\le 5$.

In this case, a generic form $f$ does have a Gram tensor of every rank in the Pataki interval as we will see in \cref{rem:ternary_summary}.
\end{example}

\section{Introduction to ternary quartics}
\label{sec:introduction to ternary quartics}

For the rest of this article, we fix $n=3$. We use the notation $A=\C[x,y,z],\,\ul x=(x,y,z),\, \Rx=\R[x,y,z]$ and write $\Sigma:=\Sigma_{3,4}$.

Since we will do some calculations in this section, we also sometimes fix a basis of $\Rx_2$ and write Gram tensors as Gram matrices wrt to this basis. Whenever we do so, we use the monomial basis ordered as follows: $x^2,y^2,z^2,xy,xz,yz$.
Let $f\in \Rx_4$ with $f=\sum_{\alpha\in\Z_+^3,\vert\alpha\vert=4} c_\alpha \ul x^\alpha$ and $c_\alpha\in\R$ for $\alpha\in\Z_+^3,\vert\alpha\vert=4$, then the Gram matrices of $f$ are given by
\[
\begingroup
\renewcommand*{\arraystretch}{1.4}
\begin{pmatrix}
c_{400} & \lambda_1 & \lambda_2 & \frac{1}{2}c_{310} & \frac{1}{2}c_{301} & \lambda_4\\
\lambda_1 & c_{040} & \lambda_3 & \frac{1}{2}c_{130} & \lambda_5 & \frac{1}{2}c_{031}\\
\lambda_2 & \lambda_3 & c_{004} & \lambda_6 & \frac{1}{2}c_{103} & \frac{1}{2}c_{013}\\
\frac{1}{2}c_{310} & \frac{1}{2} c_{130} & \lambda_6 & c_{220}-2\lambda_1 & \frac{1}{2}c_{211}-\lambda_4 & \frac{1}{2}c_{121}-\lambda_5\\
\frac{1}{2} c_{301} & \lambda_5 & \frac{1}{2} c_{103} & \frac{1}{2}c_{211}-\lambda_4 & c_{202}-2\lambda_2 & \frac{1}{2} c_{112}-\lambda_6\\
\lambda_4 & \frac{1}{2}c_{031} & \frac{1}{2}c_{013} & \frac{1}{2}c_{121}-\lambda_5 & \frac{1}{2}c_{112}-\lambda_6 & c_{022}-2\lambda_3
\end{pmatrix}
\endgroup
\]
for choices of $\lambda_1,\dots,\lambda_6\in\C$. Denote by $X$ the row vector $(x^2,y^2,z^2,xy,xz,yz)$ containing the six monomials of degree 2. Then by definition $XGX^T=f$ for every Gram matrix $G$ of $f$.

For a positive definite ternary quartic, the dimension of its \gs\ is 6 and wrt $X$ is given by all matrices of the form above that are psd.

Next, we give an introduction to bitangents and determinantal representations of ternary quartics. A modern and more abstract point of view can be found in \cite{dolgachev2012}. We mostly use a similar notation to \cite{psv2011}.

\begin{dfn}
Let $f\in \Rx_4$ be smooth and let $L\subset\P^2$ be a line. Then $L$ is called a \todfn{bitangent} of $f$, if all intersection points of $\vc(f)$ and $L$ have even multiplicity.
\end{dfn}

We refer to \cite{fulton1969} for an introduction to intersection numbers of plane curves.

An important tool for the main proof of \cref{sec:Steiner graph} is Noether's $AF+BG$ Theorem, a special version of Lasker's Theorem. A proof as well as several historical notes can be found in \cite{egh1996} or \cite{fulton1969}.

\begin{thm}[Noether's $AF+BG$ Theorem, {\cite[section 5.5.]{fulton1969},\cite[Theorem 8]{egh1996}}]
\label{thm:afbg}
Let $f,g\in\C[\ul x]$ be forms of degrees $d_1,d_2$ such that $\vc(f,g)$ is finite, and let $h\in\C[\ul x]_d$ with $d\ge d_1,d_2$. If for every point $P\in\vc(f,h)$, the intersection multiplicity of $f$ and $h$ in $P$ is at least the intersection multiplicity of $f$ and $g$ in $P$, then there exist $a\in\C[\ul x]_{d-d_1}$ and $b\in\C[\ul x]_{d-d_2}$ such that $h=af+bg$.
\end{thm}

\begin{rem}
Let $L=\vc(l)$ for some linear form $l\in A_1$. By \cref{thm:afbg} $L$ is a bitangent of $f$ if and only if there exist $q\in A_2$ and $h\in A_3$ such that $f=q^2+hl$.

Most of the time, we work with such an equation rather than considering bitangents as subsets of $\P^2$.
\end{rem}

\begin{notation}
Let $L=\vc(l),\, l\in A_1$, be a bitangent of $f$. For convenience, we also call the linear form $l$ a bitangent of $f$, as well as its projective point $[l]\in\P A_1$.

We then say that two bitangents $l_1,l_2\in A_1$ are different if $\vc(l_1),\vc(l_2)\subset\P^2$ are different lines.

Especially, when counting bitangents, this should be understood as counting lines $L\subset\P^2$, and not counting forms $l\in A_1$, as every non-zero scalar multiple of such a linear form is also a bitangent.
\end{notation}

In 1834, Pl\"ucker showed that every smooth ternary quartic has exactly 28 different bitangents. As was more recently shown in \cite{lehavi2005}, the 28 bitangents uniquely determine the smooth quartic $f\in A_4$ up to a non-zero scalar.

\begin{thm}[{\cite{pluecker1834}}]
If $f\in \Rx_4$ is smooth, then $f$ has exactly 28 different bitangents.
\end{thm}

Using the 28 bitangents of a smooth form $f\in\Sigma$, it was shown in \cite{psv2011} that one can calculate all \ldr s of $f$, as well as all eight sos representations of $f$ as a sum of three squares.

\begin{dfn}
Let $f\in \Rx_4$. Then a \todfn{\ldr} of $f$ is a matrix $M=Ax+By+Cz$ such that $f=\det(M)$, and $A,B,C$ are symmetric $4\times 4$ matrices over $\C$.
\end{dfn}

Let $f\in \Rx_4$, then two \ldr s of $f$ are called \todfn{equivalent} if they are conjugate to each other under the action of $\gl_4(\C)$.

\begin{thm}[{\cite{hesse1855}}]
Every smooth quartic $f\in \Rx_4$ has exactly 36 inequivalent \ldr s.
\end{thm}

\begin{dfn}
Let $f\in \Rx_4$ be a smooth quartic and fix one \ldr\ $M=Ax+By+Cz$ of $f$. The \todfn{Cayley octad of $M$} is defined as the eight solutions $O_1,\dots,O_8\in\P^3$ of the system $uAu^T=uBu^T=uCu^T=0$ with $u=(u_0:\dots:u_3)\in\P^3$.
\end{dfn}

Intersecting three quadratic hypersurfaces in $\P^3$, we expect there to be exactly eight different intersection points $O_1,\dots,O_8$. Using the fact that $f$ is smooth this is indeed true as proven in \cite[Proposition 6.3.3.]{dolgachev2012}.

This enables us to enumerate the bitangents of $f$ in the following way.

\begin{prop}[{\cite[Proposition 3.3.]{psv2011}}]
Let $f\in \Rx_4$ be a smooth quartic. Let $M$ be a \ldr\ of $f$ with Cayley octad $O_1,\dots,O_8$. Then the 28 bitangents of $f$ are given by $L_{ij}$ with $L_{ij}=\vc(O_iMO_j^T)$, $1\le i<j\le 8$.
\end{prop}

\begin{rem}
(i) Since $M$ is symmetric, we have $O_iMO_j^T=O_jMO_i^T$ for every $i,j\in\{1,\dots,8\}$, $i\neq j$. Thus, for every $i,j\in\{1,\dots,8\}$, $i\neq j$, the line $L_{ij}=\vc(O_iMO_j^T)=L_{ji}$ is a bitangent of $f$ and we do not have to pay attention to the order of the indices.

(ii) With a fixed Cayley octad, we write $b_{ij}\in A_1$ ($1\le i < j\le 8,\, i\neq j$) for the bitangents of $f$ where $L_{ij}=\vc(b_{ij})$ and set $b_{ij}:=b_{ji}$ for all $1\le j<i\le 8,\, i\neq j$. This is mostly for convenience to avoid unnecessary scaling later on.

(iii) This enumeration of the bitangents is dependent on the choice of a \ldr\ of $f$ and on the order of the Cayley octad.

(iv) By \cite[\S 0]{vinnikov1993} every smooth real quartic has a \textit{real} \ldr. The Cayley octad corresponding to this \ldr\ contains only real points or points appear in complex conjugate pairs.
\end{rem}

\begin{rem}
Recall that we defined Gram tensors not only as the real tensors that are mapped to a form via the Gram map, but also the complex ones.

By Gram tensor we always mean a possibly complex one, and where necessary we say real Gram tensor to specify.
\end{rem}

Next, we follow the algorithm in \cite[\S 5]{psv2011} to construct all 63 Gram tensors of $f$ and to determine combinatorially from a Cayley octad which are real.

\begin{dfn}
Let $f\in \Rx_4$. Then a \todfn{\qdr} (QDR) of $f$ is a matrix 
$Q=\begin{pmatrix}
q_0 & q_1\\
q_1 & q_2
\end{pmatrix}$
with $q_0,q_1,q_2\in A_2$ such that $f=\det(Q)$.
\end{dfn}

\begin{prop}[{\cite[section 254]{salmon1879},\cite[Proposition 5.7.]{psv2011}}] 
\label{prop:contact_conics}
Let $f\in \Rx_4$ be a smooth quartic and let $Q$ be a QDR of $f$. Then the variety $\{\lambda Q \lambda^T\colon \lambda\in\P^1\}\subset \P A_2$ contains exactly six products of two bitangents of $f$. I.e. there exist twelve bitangents $l_1^{},l_1',\dots,l_6^{},l_6'\in A_1$ of $f$ such that $[l_i^{}l_i']\in \lambda Q\lambda^T$.
\end{prop}

\begin{lab}
\label{lab:qdr_to_sos}
Let $f\in \Rx_4$ be a smooth quartic, and let $Q=\begin{pmatrix}
q_0 & q_1\\
q_1 & q_2
\end{pmatrix}$
be a QDR of $f$.

This means $f=q_0q_2-q_1^2$ which gives rise to a sos representation of $f$ over $\C$ as follows
\[
f=\left(\frac{1}{2}q_0+\frac{1}{2}q_2\right)^2+\left(\frac{1}{2i}q_0-\frac{1}{2i}q_2\right)^2+(iq_1)^2.
\]
We call the Gram tensor $\theta$ corresponding to this sos representation of $f$ the Gram tensor corresponding to $Q$.
Then $\im_\C\theta=\spn(q_0,q_1,q_2)$. 
$\theta$ is well-defined: let $f=p_0p_2+p_1^2$ ($p_0,p_1,p_2\in A_2$) be any other representation of $f$ such that $\spn(p_0,p_1,p_2)=\spn(q_0,q_1,q_2)$. Then the Gram tensor corresponding to this equation has the same image as $\theta$. As the map $\sy{\im_\C(\theta)}\stackrel{\mu}{\to} A_4$ is injective, we see that there is only one Gram tensor with this image, hence the two are equal.
\end{lab}

Corresponding to every \qdr\ of a form $f$, we cannot only associate a Gram tensor of $f$ as above, but also a Steiner complex. These are also in one-to-one correspondence with \qdr s of $f$ but are of a more combinatorial nature.

\begin{thm}
\label{thm:equivalences_sc}
Let $f\in \Rx_4$ be a smooth quartic, and let $l_1^{},l_1',\dots,l_6^{},l_6'\in A_1$ be twelve different bitangents of $f$. Write $\mS=\{\{l_i^{},l_i'\}\colon i=1,\dots,6\}$, then the following are equivalent:
\begin{enumerate}
\item The six quadrics $[l_1^{}l_1'],\dots,[l_6^{}l_6']$ are on the hypersurface $\lambda Q\lambda^T,\, \lambda\in\P^1$, for a \qdr\ $Q$ of $f$.
\item For ever $i\neq j$, there exist $q\in A_2$ and $\lambda\in\C$ such that $f=\lambda l_i^{}l_i'l_j^{}l_j'+q^2$, i.e. the intersection points of $\vc(l_i^{}l_i'l_j^{}l_j')$ and $\vc(f)$ lie on a conic $\vc(q)$.
\item Let $M$ be a \ldr\ of $f$ and \\$O_1,\dots,O_8$ the corresponding Cayley octad. Then
\begin{align*}
\mS &=\{ \{b_{ik},b_{jk}\}\colon \{i,j\}=I, k\in I^c  \} \text{ for some } I\subset \{1,\dots,8\},\, |I|=2 \text{ or}\\
\mS &=\{ \{b_{ij},b_{kl}\}\colon \{i,j,k,l\}=I \text{ or } \{i,j,k,l\}=I^c  \} \text{ for some } I\subset \{1,\dots,8\},\, |I|=4.
\end{align*}
\end{enumerate}
($I^c=\{1,\dots,8\}\setminus I$ denotes the complement of $I$ in $\{1,\dots,8\}$)
Any sextuple $\mS$ satisfying the equivalent conditions (i)-(iii), is called a \todfn{Steiner complex} of $f$.
\end{thm}

From this, we can also calculate the number of Steiner complexes. There are $\binom{8}{2}=28$ Steiner complexes of the first type and $\frac{1}{2}\cdot\binom{8}{4}=35$ of the second type. 
In the second case we get the same Steiner complex if we choose $I$ or $I^c$. 

We say that two \qdr\ are equivalent if the images of the corresponding Gram tensors are the same.
Up to this equivalence there are exactly 63 \qdr\ of a smooth quartic $f$ and these are in one-to-one correspondence with the 63 Steiner complex.

After fixing a \ldr\ and a Cayley octad of $f$, the rank 3 Gram tensors of $f$ can be enumerated in the same way the Steiner complexes can, i.e. by subsets of $\{1,\dots,8\}$ as above.

For ease of notation, we write $I=ijkl$ instead of $I=\{i,j,k,l\}\subset\{1,\dots,8\}$, $|I|=4$, and $I=ij$ instead of $I=\{i,j\}\subset\{1,\dots,8\}$, $|I|=2$.

There was already some computational evidence found by Powers and Reznick \cite{pr2000} that a smooth psd ternary quartic has exactly 15 real Gram tensors of rank 3. Later, in \cite{prss2004} this was proven using the connection of \qdr\ with 2-torsion points of the corresponding curve.
In \cite{psv2011} the authors found a way to identify these representations using the combinatorial properties of Steiner complexes.

Let $f\in \Sigma$ be a smooth quartic, and let $M$ be a \textit{real} \ldr\ of $f$. Since $f$ is smooth, the curve defined by $f$ contains no real points, i.e. $\vc(f)(\R)=\emptyset$. Using table 1 in \cite{psv2011}, this means that there is no real point in the Cayley octad, hence it consists of four conjugate pairs.
After reordering the Cayley octad, we may assume that $\ol O_i=O_{i+1}$ for $i=\{1,3,5,7\}$ where $\ol{\phantom{\rule{0.1mm}{0.1mm}}\cdot\phantom{\rule{0.1mm}{0.1mm}}}$ denotes complex conjugation. Then the real Gram tensors of rank 3 of $f$ are given as follows.

\begin{thm}[{\cite[Theorem 6.2. \& 6.3.]{psv2011}}]
\label{thm:psd_gram_tensors}
The eight real psd Gram tensors of rank 3 of $f$ correspond to the following eight Steiner complexes:
\begin{align*}
&1357,\quad 1368,\quad 1458,\quad 1467,\\
&1358,\quad 1367,\quad 1457,\quad 1468.
\end{align*}
The other seven real Gram tensors of rank 3 correspond to the Steiner complexes given by
\begin{gather*}
1234,\quad 1256,\quad 1278,\\
12,\quad 34,\quad 56,\quad 78.
\end{gather*}
\end{thm}

\section{The Steiner graph}
\label{sec:Steiner graph}

\begin{dfn}
Let $f\in\Sigma$ be smooth. Consider the graph whose vertices correspond to the eight real psd rank 3 Gram tensors of $f$ (or equivalently their Steiner complexes). We draw an edge between two vertices if the line segment connecting the two corresponding Gram tensors on $\gram(f)$ is contained in the boundary of $\gram(f)$.
This graph is called the \todfn{Steiner graph} of $f$.
\end{dfn}

We show that the Steiner graph of a smooth psd quartic is the union of two $K_4$, where $K_4$ denotes the complete graph on four vertices.
This graph has already been considered in \cite{psv2011} where it is shown that it contains the disjoint union of two $K_4$.
The fact that there are no additional edges in the Steiner graphs seems to be claimed as well in the paper, but without proof.
In this section, we show that this is indeed the case.
For this, we need some more facts about possible arrangements of bitangents which we show first.

A similar graph was constructed and studied for binary forms in \cite{scheiderer2018}.

\begin{rem}
\label{rem:steiner_graph}
We show that for a smooth quartic $f\in\Sigma$ the Steiner graph of $f$ has the form in \cref{fig:steiner_graph}. It is the disjoint union of two complete graphs on four vertices and the vertices in them are given by the four Steiner complexes in a row in \cref{thm:psd_gram_tensors}.

The proof that the graph below is contained in the Steiner graph is given in \cref{lem:Steiner_graph_1} which is taken from \cite[Lemma 6.4.]{psv2011}.

\begin{figure}[h]
\centering
\begin{tikzpicture}[main/.style = {draw, ellipse}] 
	\node[main] (1) at (0,0) {$1357$}; 
	\node[main] (2) at (1.2,-1.2) {$1368$};
	\node[main] (3) at (1.2,1.2) {$1458$}; 
	\node[main] (4) at (2.4,0) {$1467$};
	\draw (1) -- (2);
	\draw (1) -- (3);
	\draw (1) -- (4);
	\draw (2) -- (3);
	\draw (2) -- (4);
	\draw (3) -- (4);
	
	\node[main] (5) at (4.8,0) {$1358$}; 
	\node[main] (6) at (6,-1.2) {$1367$};
	\node[main] (7) at (6,1.2) {$1457$}; 
	\node[main] (8) at (7.2,0) {$1468$};
	\draw (5) -- (6);
	\draw (5) -- (7);
	\draw (5) -- (8);
	\draw (6) -- (7);
	\draw (6) -- (8);
	\draw (7) -- (8);
\end{tikzpicture}
\caption{The Steiner graph of a smooth quartic.}
\label{fig:steiner_graph}
\end{figure}
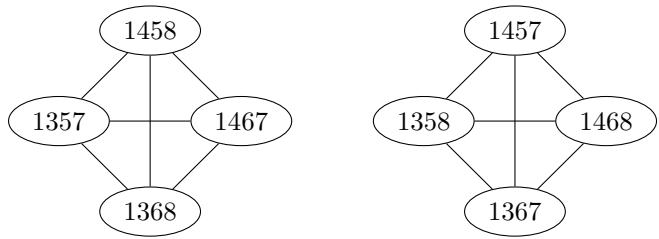
\end{rem}

We fix a smooth quartic $f\in\Sigma$, a \ldr\ of $f$, and a Cayley octad.

\begin{prop}
\label{prop:intersection_sc}
Let $\mS_1,\,\mS_2$ be two different Steiner complexes. Then one of the following holds:
\begin{enumerate}
\item There exist four different bitangents $l_1,\dots,l_4$ such that $\{l_1,l_2\},\{l_3,l_4\}\in\mS_1$ and $\{l_1,l_3\},$ $\{l_2,l_4\}\in\mS_2$, or
\item there exist six bitangents $l_1,\dots,l_6$ such that $\{l_i^{},l_i'\}\in\mS_1$ and $\{l_i^{},l_i''\}\in\mS_2$ for every $i=1,\dots,6$ where $l_i^{},l_i',l_i''$, $i=1,\dots,6$ are 18 different bitangents.
\end{enumerate}
\end{prop}
\begin{proof}
This follows from \cite[Lemma 5.4.8., Proposition 6.1.6.]{dolgachev2012}.
\end{proof}

\begin{dfn}
In the situation of \cref{prop:intersection_sc}, the two Steiner complexes are called \todfn{syzygetic} if they satisfy (i) and \todfn{azygetic} if they satisfy (ii).
\end{dfn}

With a fixed \ldr\ and a Cayley octad, being syzygetic or azygetic for a pair of Steiner complexes translates into the following statements about subsets of $\{1,\dots,8\}$ representing the Steiner complexes.

\begin{lem}
\label{lem:steiner_complexes_to_subsets}
Let $I,J\subset \{1,\dots,8\}$ be subsets such that $\vert I\vert,\vert J\vert=4$ and $I\neq J,J^c$. Let $\mS_1,\mS_2$ be the corresponding Steiner complexes. Then $\mS_1$ and $\mS_2$ are syzygetic if and only if $\vert I\cap J\vert=2$.
\end{lem}
\begin{proof}
See \cite[Lemma 6.5]{psv2011}.
\end{proof}

\begin{example}
We consider the Steiner complexes $1358$ and $1457$ in the Steiner graph. These form a syzygetic pair. Indeed, by \cref{lem:steiner_complexes_to_subsets} we only have to check that the sets $1358$ and $1457$ contain exactly two common elements. This is true and the elements are $1$ and $5$.

Moreover, we see from the proof of \cref{lem:steiner_complexes_to_subsets} that the four bitangents in \cref{prop:intersection_sc} (i) are given by $b_{15},b_{38},b_{47},b_{26}$. The pairs $\{b_{15},b_{38}\},\,\{b_{47},b_{26}\}$ are contained in $1358$ and $\{b_{15},b_{47}\},\,\{b_{38},b_{26}\}$ are contained in $1457$.
It is now natural to consider the Steiner complex containing $\{b_{15},b_{26}\},\,\{b_{47},b_{38}\}$ which is the third possibility to form two pairs from the four bitangents. This Steiner complex is therefore $1256$.
We note that it does correspond to a real rank 3 tensors which is not psd by \cref{thm:psd_gram_tensors}.
\end{example}

We can now give the proof that the Steiner graph of $f$ contains all the edges shown in \cref{rem:steiner_graph}.

\begin{lem}[{\cite[Lemma 6.4.]{psv2011}}]
\label{lem:Steiner_graph_1}
Let $\mS_1,\mS_2$ be two Steiner complexes corresponding to real psd Gram tensors, that are in the same line in \cref{thm:psd_gram_tensors}. Then the line between the corresponding Gram tensors on $\gram(f)$ is contained in the boundary of $\gram(f)$.

I.e. if $\theta_1,\theta_2$ are the corresponding Gram tensors, then for every $\lambda\in [0,1]$ the point $\lambda\theta_1+(1-\lambda)\theta_2$ is on the boundary of $\gram(f)$.
\end{lem}
\begin{proof}
By \cref{lem:steiner_complexes_to_subsets} the two Steiner complexes $\mS_1,\,\mS_2$ are syzygetic, hence there exist bitangents $l_1,\dots,l_4$ such that $\{l_1,l_2\},\{l_3,l_4\}\in\mS_1$ and $\{l_1,l_3\},\{l_2,l_4\}\in\mS_2$. 
Let $q\in A_2$, $\lambda\in\C$ such that $f=\lambda l_1l_2l_3l_4+q^2$, which is possible by \cref{thm:equivalences_sc}. Then $q$ is contained in the (complex) images of $\theta_1$ and $\theta_2$. Hence $\dim(\im_\C(\theta_1)+\im_C(\theta_2))\le 5$. Since $\theta_i$ is real, $\im_\C(\theta_i)$ has a real $\C$-basis ($i=1,2$), thus it follows that $\dim(\im(\theta_1+\theta_2))=\dim\left(\im(\theta_1)+\im(\theta_2)\right)\le 5$.
\end{proof}

We now work towards showing that the Steiner graph of a smooth quartic is the disjoint union of two $K_4$. We start with some facts about bitangents.

\begin{dfn}
Let $f\in\Rx_4$ be smooth, and let $l_1,l_2,l_3$ be three different bitangents of $f$. The triple is called \todfn{syzygetic} if there exists a Steiner complex $\mS$ such that $\{l_1,l_2\},\{l_3,l\}\in\mS$ for another bitangent $l$ of $f$. Otherwise, the triple is called \todfn{azygetic}.
Analogously, a 4-tuple $l_1,\dots,l_4$ of bitangents is called syzygetic if one (equivalently any) subset consisting of three bitangents in syzygetic.
\end{dfn}

\begin{lem}
\label{lem:azygetic_if_in_sc}
Let $l_1,l_2,l_3$ be three bitangents. The following are equivalent:
\begin{enumerate}
\item There exist bitangents $l_1',l_2',l_3'$ such that $l_1^{},l_2^{},l_3^{},l_1',l_2',l_3'$ are pairwise different, and there exists a Steiner complex $\mS$ such that\\ $\{l_1^{},l_1'\},\{l_2^{},l_2'\},\{l_3^{},l_3'\}\in\mS$.
\item The bitangents $l_1,l_2,l_3$ are azygetic.
\end{enumerate}
\end{lem}
\begin{proof}
(i)$\Rightarrow$(ii): \cite[Lemma 6.1.6.]{dolgachev2012}. (ii)$\Rightarrow$(i): \cite[Proposition 6.1.4.]{dolgachev2012}.
\end{proof}

\begin{lem}
\label{lem:azy_bit_no_intersection}
Three azygetic bitangents cannot intersect in a common point.
\end{lem}
\begin{proof}
Let $l_1,l_2,l_3$ be three azygetic bitangents. By \cref{lem:azygetic_if_in_sc} we can find a Steiner complex $\mS$ such that $\{l_1^{},l_1'\},\{l_2^{},l_2'\},\{l_3^{},l_3'\}\in\mS$ for some bitangents $l_1',l_2',l_3'$.

The image of the corresponding Gram tensor is then given by $\spn(l_1^{}l_1',l_2^{}l_2',l_3^{}l_3')$.
If $l_1,l_2,l_3$ would intersect in a point, then this space has a \bp\ and thus $f$ has a singularity at this point.
\end{proof}

\begin{rem}
Summarizing some facts about syzygetic and azygetic bitangents and Steiner complexes.

Let $l_1,l_2,l_3$ be three syzygetic bitangents. Then there is a unique fourth bitangent $l_4$ and three (necessarily) syzgetic Steiner complexes $\mS_1,\mS_2,\mS_3$ such that 
\[
\{l_1,l_2\},\{l_3,l_4\}\in\mS_1,\, \{l_1,l_3\},\{l_2,l_4\}\in\mS_2,\, \text{ and }\, \{l_1,l_4\},\{l_2,l_3\}\in\mS_3.
\]
Furthermore, any two syzygetic Steiner complexes determine a unique third one, defined as above.

Let $l_1,l_2,l_3$ be three azygetic bitangents. Then, these three cannot intersect in a common point, and they are contained in a Steiner complex such that each of the three bitangents belongs to a different pair in the Steiner complexes (see \cref{lem:azygetic_if_in_sc}). This Steiner complex is not unique, and any two such azygetic Steiner complexes have exactly six common bitangents where each bitangent belongs to a different pair in each Steiner complex.
\end{rem}

\begin{example}
Consider the two different Steiner complexes $\{\{l_i^{},l_i'\}\colon i=1,\dots,6\}$ and $\{\{l_i^{},l_i''\}\colon i=1,\dots,6\}$. Then these are azygetic and have exactly the six bitangents $l_1,\dots,l_6$ in common.
These could for example be the Steiner complexes corresponding to the subsets $I=1357$ and $J=1358$ after fixing a Cayley octad and a \ldr.
\end{example}

We need one last simple lemma before turning to the main theorem of this section.

\begin{lem}
\label{lem:orth_change}
Let $n\ge 2,\, d\ge 1$. Let $f\in\Sigma_{n,2d}$ and let $\theta\in\gram(f)$ with corresponding subspace $U=\im(\theta)$ of dimension $r$. If $p_1,\dots,p_r\in U$ form a basis of $U$, then there exist $a_{ij}\in\R$ ($1\le i\le r,\, 1\le j\le i$) such that
\[
f=\sum_{i=1}^r \left(\sum_{j=1}^i a_{ij} p_j\right)^2.
\]
I.e. the $i$-th summand is a linear combination of only the first $i$ forms $p_1,\dots,p_i$.
\end{lem}
\begin{proof}
This can be shown using the QL-decomposition of a matrix. Let $f_1,\dots,f_r\in U$ be a basis such that $f=\sum_{i=1}^r f_i^2$. Let $A$ be the $r\times r$ matrix containing the coordinates of $f_i$ wrt the basis $p_1,\dots,p_r$ in its $i$-th row. Then
\[
f=(f_1,\dots,f_r)(f_1,\dots,f_r)^T= (p_1,\dots,p_r)A^TA(p_1,\dots,p_r)^T.
\]
As $A$ is a rank $r$ square matrix there exists an orthogonal matrix $O\in\orth(r)$ and a lower triangular $r\times r$ real matrix $L$ such that $A=OL$. This shows
\[
f=(p_1,\dots,p_r)(OL)^T(OL)(p_1,\dots,p_r)^T=(p_1,\dots,p_r)L^TL(p_1,\dots,p_r)^T.
\]
Since $L$ is lower triangular, this sos representation of $f$ has the required form.
\end{proof}

We now show the following theorem which is our goal in this section.

\begin{thm}
\label{thm:steiner_graph_non_deg}
Let $f\in \Sigma$ be a smooth quartic. Then the Steiner graph of $f$ is the union of two disjoint $K_4$.
\end{thm}

We already know that there are two disjoint $K_4$ contained in the Steiner graph, hence we need to show that there are no additional edges. We first show that none of those edges can have rank $\le 4$.

For the rest of this section, we fix the following setup. Let $f\in \Sigma$ smooth and let $\theta_1,\,\theta_2\in\gram(f)$ be two real psd rank 3 Gram tensors with corresponding Steiner complexes $\mS_1,\,\mS_2$ such that $\mS_1$ and $\mS_2$ are azygetic. By \cref{lem:Steiner_graph_1} these are the Gram tensors left to consider. Equivalently, syzygetic Steiner complexes in the Steiner graph are connected by an edge and azygetic ones are not.

\begin{prop}
\label{prop:steiner_graph_deg_no_rk_4}
$\dim \left(\im(\theta_1)+\im(\theta_2)\right)\ge 5$.
\end{prop}
\begin{proof}
Since $\mS_1$ and $\mS_2$ are azygetic, we can pick pairs of bitangents such that
\[
\{l_1^{},l_1'\},\{l_2^{},l_2'\},\{l_3^{},l_3'\}\in\mS_1 \text{ and }\{l_1^{},l_1''\},\{l_2^{},l_2''\}\in\mS_2.
\]
Assume $\dim (\im(\theta_1)+\im(\theta_2))\le 4$. Then the quadratic forms $l_1^{}l_1',l_2^{}l_2',l_3^{}l_3',l_1^{}l_1'',l_2^{}l_2''$ are linearly dependent. We therefore find $\lambda_1,\dots,\lambda_5\in\C$ not all zero such that
\begin{align*}
0&=\lambda_1 l_1^{}l_1'+\lambda_2 l_2^{}l_2'+\lambda_3 l_3^{}l_3' +\lambda_4 l_1^{}l_1''+\lambda_5 l_2^{}l_2''\\
&=l_1(\lambda_1 l_1'+\lambda_4 l_1'')+ l_2 (\lambda_2 l_2'+\lambda_5 l_2'') + \lambda_3 l_3^{}l_3'.
\end{align*}
Let $P\in\P^2$ be the intersection point $\vc(l_1,l_2)$. Then either $l_3$ or $l_3'$ has to vanish at $P$ if $\lambda_3\neq 0$. But both triples $l_1,l_2,l_3$ and $l_1^{},l_2^{},l_3'$ are azygetic by \cref{lem:azygetic_if_in_sc}, hence cannot intersect in a common point by \cref{lem:azy_bit_no_intersection}. This means that $\lambda_3=0$ which shows
\[
0=l_1(\lambda_1 l_1'+\lambda_4 l_1'')+ l_2 (\lambda_2 l_2'+\lambda_5 l_2'').
\]
Since two bitangents cannot be multiples of each other, $l_1\in\spn(l_2',l_2'')$ and $l_2\in\spn(l_1',l_1'')$. But again the triples $l_1^{},l_2',l_2''$ and $l_2^{},l_1',l_1''$ are azygetic by the next lemma and therefore cannot intersect in a common point.
\end{proof}

\begin{lem}
With everything as in \cref{prop:steiner_graph_deg_no_rk_4}, the triples $l_1^{},l_2',l_2''$ and $l_2^{},l_1',l_1''$ of bitangents are azygetic.
\end{lem}
\begin{proof}
Consider the Steiner complex $\mS_1=\{\{l_i^{},l_i'\}\colon i=1,\dots,6\}$ and the four bitangents $l_1^{},l_1',l_2^{},l_2'$.  Then, there exist Steiner complexes $\mT_1,\mT_2$ such that $\mT_1$ contains $\{l_1,l_2\},\{l_1',l_2'\}$ and $\mT_2$ contains $\{l_1^{},l_2'\},\{l_1',l_2^{}\}$. By \cite[Theorem 5.4.10.]{dolgachev2012} any of the 28 bitangents of $f$ is contained in one the three Steiner complexes. Especially, one contains the bitangent $l_1''$.
This bitangent cannot be contained in $\mT_2$: otherwise $\mT_2$ contains $\{l_1^{},l_2'\},\{l_1',l_2^{}\}$ and $\{l_1'',L\}$ for some bitangent $L$. Especially $l_1^{},l_2^{},l_1''$ are azygetic by \cref{lem:azygetic_if_in_sc}.
But $\mS_2$ contains $\{l_1^{},l_1''\}$ and $\{l_2^{},l_2''\}$. Hence $l_1^{},l_2^{},l_1''$ are syzygetic, a contradiction.

Therefore $l_1''$ is contained in the Steiner complex $\mT_1$. Then $l_1',l_1'',l_2^{}$ are contained in $\mT_1$ as in \cref{lem:azygetic_if_in_sc}, hence the triple is azygetic.

An analogous argument shows that $l_1^{},l_2',l_2''$ form an azygetic triple.
\end{proof}

\begin{rem}
We first sketch the second part of the proof of \cref{thm:steiner_graph_non_deg}. If $\dim \left(\im_\C(\theta_1)+\im_\C(\theta_2)\right)= 5$ and $q\in A_2$ spans the intersection of the images, then we get representations $p_1p_2+q^2=cf$ and $q_1q_2+q^2=df$ for some $p_1,p_2\in\im_\C(\theta_1)$, $q_1,q_2\in\im_\C(\theta_2)$ and $c,d\in\C$. 

Let $P_1,\dots,P_8$ be the intersection points of $\vc(q)$ and $\vc(f)$. The equations show that the products $q_1q_2$ and $p_1p_2$ are tangent to $f$ at the points $P_1,\dots,P_8$. 
A form $g\in\P A_2$ has 5 degrees of freedom and we require $\vc(g)$ to be tangent at 4 points. Therefore, it seems unlikely to find two different products of two quadratic forms satisfying the conditions.

In the case where the two Steiner complexes $\mS_1,\,\mS_2$ are syzygetic, this is easily possible: as in \cref{lem:Steiner_graph_1} there can be four bitangents $l_1,l_2,l_3,l_4\in A_1$ of $f$ such that $q_1=l_1l_2$, $q_2=l_3l_4$ and $p_1=l_1l_3$, $p_2=l_2l_4$.
\end{rem}

We prepare with an application of \cref{thm:afbg}.

Assume we have an equation $f=p_1p_2+q^2$ for some $p_1,p_2,q\in A_2$. Let $\vc(f,p_1)=\{P_1,\dots,P_4\}$ and $\vc(f,p_2)=\{P_5,\dots,P_8\}$ with $P_1,\dots,P_8\in\P^2$ (not necessarily different). 
For any two different points $P,Q\in\P^2$, we write $\ol{PQ}$ for the line through $P$ and $Q$. Let
\[
L_1=\ol{P_1P_2},\,L_2=\ol{P_3P_4},\,L_3=\ol{P_5P_6},\,L_4=\ol{P_7P_8}
\]
if all points are different. If any two points $P_i,P_j$ are the same, the line $\ol{P_iP_j}$ should be understood as the tangent line at $P_i$ of $\vc(p_1)$ or $\vc(p_2)$ depending on whether $1\le i\le 4$ or $5\le i \le 8$. 
Moreover, for $i=1,\dots,4$ write $l_i\in A_1$ for the linear form such that $\vc(l_i)=L_i$.

\begin{prop}
\label{prop:afbg_representation}
In the situation above, if $q$ is also irreducible, there exists $h\in \spn(p_1,p_2,q)\subset A_2$ and $0\neq\lambda\in\C$ such that
\[
f=\lambda l_1l_2l_3l_4+qh.
\]
\end{prop}
\begin{proof}
By \cref{thm:afbg}, there exist $\lambda_1,\dots,\lambda_6\in\C$ such that
\[
\lambda_1 p_1=\lambda_2 l_1l_2+\lambda_3 q,\quad \lambda_4 p_2=\lambda_5 l_3l_4+\lambda_6 q
\]
by choice of the forms $l_1,\dots,l_4$.
If $\lambda_2\lambda_5=0$, then either $q=\nu p_1$ or $q=\nu p_2$ for some $\nu\in\C$. However, the image of the Gram tensor corresponding to the representation $f=p_1p_2+q^2$ has image $\spn(p_1,p_2,q)$. Hence, this Gram tensor has rank 2 which is not possible since $f$ is smooth.
If $\lambda_1\lambda_4=0$, then $q$ was reducible, contradicting the assumption.

We therefore see $\spn(p_1,p_2,q)=\spn(l_1l_2,l_3l_4,q)$. Multiplying the two equations, we get
\[
\lambda_1\lambda_4p_1p_2=\lambda_2\lambda_5l_1l_2l_3l_4+q(\lambda_3\lambda_5l_3l_4+\lambda_3\lambda_6q+\lambda_2\lambda_6l_1l_2).
\]
Substituting with the equation $f=p_1p_2+q^2$, we get
\[
\lambda_1\lambda_4(f-q^2)=\lambda_2\lambda_5l_1l_2l_3l_4+q(\lambda_3\lambda_5l_3l_4+\lambda_3\lambda_6q+\lambda_2\lambda_6l_1l_2),
\]
and thus
\[
f=\frac{\lambda_2\lambda_5}{\lambda_1\lambda_4}l_1l_2l_3l_4+q\underbrace{\left(\frac{\lambda_3\lambda_5}{\lambda_1\lambda_4}l_3l_4+\frac{\lambda_3\lambda_6}{\lambda_1\lambda_4}q+\frac{\lambda_2\lambda_6}{\lambda_1\lambda_4}l_1l_2+q\right)}_{h:=}.
\]
\end{proof}

\begin{proof}[{Proof of \cref{thm:steiner_graph_non_deg}}]
Let $\spn(q)=\im(\theta_1)\cap\im(\theta_2)$. By \cref{lem:orth_change} there exist $p_1,p_2\in\im_\C(\theta_1)$, $q_1,q_2\in\im_\C(\theta_2)$, and $0\neq c,d\in\R$ such that
\begin{align}
\label{eq:sc_eq_1.1}
f&=p_1p_2+(cq)^2,\\
\label{eq:sc_eq_1.2}
f&=q_1q_2+(dq)^2.
\end{align}

If $c^2=d^2\neq 0$, subtracting equations (\ref{eq:sc_eq_1.1}) and (\ref{eq:sc_eq_1.2}) yields $p_1p_2=q_1q_2$. Therefore, either $\im(\theta_1)=\im(\theta_2)$ or all these forms are reducible and we can write $p_1=l_1l_2, p_2=l_3l_4$ and $q_1=l_1l_3, q_2=l_2l_4$ for some $l_1,l_2,l_3,l_4\in A_1$. 
But the equations show that $l_1,\dots,l_4$ are bitangents and $\{l_1,l_2\},\{l_3,l_4\}$ are in the first Steiner complex and $\{l_1,l_3\},\{l_2,l_4\}$ are in the second. Hence the two Steiner complexes are syzygetic, a contradiction.

Assume $q$ is irreducible. By \cref{prop:afbg_representation} there exist $h_1\in\im_\C(\theta_1)$, $h_2\in\im_\C(\theta_2)$, and $0\neq \nu,\lambda\in\C$ such that
\begin{align}
\label{eq:sc_eq_2.1}
f&=\lambda l_1l_2l_3l_4 +qh_1,\\
\label{eq:sc_eq_2.2}
f&=\nu l_1l_2l_3l_4 +qh_2.
\end{align}
Subtracting the two equations, we get 
\[
(\nu-\lambda)l_1l_2l_3l_4=q(h_1-h_2).
\]
Since $q$ is irreducible, it follows that $\nu=\lambda$ and $h_1=h_2$. Hence, $h_1=h_2\in\im_\C(\theta_1)\cap\im_\C(\theta_2)=\spn(q)$ and equations (\ref{eq:sc_eq_2.1}) and (\ref{eq:sc_eq_2.2}) now read
\begin{align*}
f=\lambda l_1l_2l_3l_4 +\lambda' q^2\\
f=\nu l_1l_2l_3l_4 +\nu' q^2
\end{align*}
for some $\lambda',\nu'\in\C$.
This shows that $l_1,\dots,l_4$ are bitangents of $f$. Moreover, wlog 
\[
\{l_1,l_2\},\{l_3,l_4\}\in\mS_1\quad \text{and}\quad \{l_1,l_3\},\{l_2,l_4\}\in\mS_2,
\]
which means that $\mS_1$ and $\mS_2$ are syzygetic, a contradiction.

Now, we may assume $q$ is reducible. We write $q=l_1l_2$ for some $l_1,l_2\in A_1$. If any of the forms $p_1,p_2,q_1,q_2$ was divisible by $l_1$ or $l_2$, the form $f$ was as well. Thus all forms $p_1,p_2,q_1,q_2$ can only share two points with $l_1$ and $l_2$.

Let $\vc(f,l_1)=\{P_1,\dots,P_4\}$ and $\vc(f,l_2)=\{P_5,\dots,P_8\}$. 
Then, wlog we may assume
\begin{align}
\label{eq:point_split_1}
\vc(p_1,f)=\{P_1,P_2,P_5,P_6\},\quad &\vc(p_2,f)=\{P_3,P_4,P_7,P_8\},\\
\label{eq:point_split_2}
\vc(q_1,f)=\{P_1,P_2,P_7,P_8\},\quad &\vc(q_2,f)=\{P_3,P_4,P_5,P_6\}.
\end{align}
Firstly, $\vc(q_1)$ cannot share more than two points with any of the lines $\vc(l_1)$ and $\vc(l_2)$ due to the reasoning above. However, it may also not share more than two points with $\vc(p_i)$ ($i=1,2$) since $q_1$ and $p_i$ are both tangent to $f$ at these points. 
Otherwise, by B\'{e}zouts Theorem either $q_1\in\spn(p_i)$ which shows $\dim(\im_\C(\theta_1)+\im_\C(\theta_2))\le 4$ contradicting \cref{prop:steiner_graph_deg_no_rk_4}, or $q_1$ and $p_i$ are reducible and not multiples of each other. Again this leads to a contradiction as follows: assume wlog $i=1$, $q_1=h_1h_2$ and $p_1=h_1h_3$ for some $h_1,h_2,h_3\in A_1$. Then,
\begin{align*}
f&=h_1h_3p_2+(cq)^2,\\
f&=h_1h_2q_2+(dq)^2.
\end{align*}
Subtracting the two equations and plugging in $q=l_1l_2$, we get
\[
h_1(h_3p_2-h_2q_2)=(d^2-c^2)l_1^2l_2^2.
\]
From the beginning of the proof we know $d^2-c^2\neq 0$, hence $h_1\in\spn(l_1)$ or $h_1\in\spn(l_2)$. In either case, $f$ is reducible, a contradiction. This shows that the only way to split the points $P_1,\dots,P_8$ is the one given in (\ref{eq:point_split_1}) and (\ref{eq:point_split_2}).

Again, we use \cref{thm:afbg} and find $\lambda_1,\dots,\lambda_4\in\C$ such that
\[
p_1=\lambda_1l_1^2+\lambda_2q_1,\quad  p_2=\lambda_3l_1^2+\lambda_4q_2.
\]
However, this means
\[
\im_\C(\theta_1)+\im_\C(\theta_2)=\spn(p_1,p_2,q_1,q_2,q)=\spn(l_1^2,q_1,q_2,q),
\]
but by \cref{prop:steiner_graph_deg_no_rk_4}, this is impossible.

To finish the proof of \cref{thm:steiner_graph_non_deg}, we note that $\im(\theta_1)+\im(\theta_2)=\im(\theta_1+\theta_2)$ since $\theta_1,\theta_2$ are real psd Gram tensors. Therefore, any tensor $\lambda \theta_1+(1-\lambda)\theta_2$ ($\lambda\in [0,1]$) has image $\im(\theta_1)+\im(\theta_2)$ which has dimension 6.
\end{proof}

\begin{rem}
We briefly sketch a second possible proof. By slightly adjusting \cref{prop:steiner_graph_deg_no_rk_4}, we can show that $\dim(\im(\theta_1)+\im(\theta_2))=5$ if and only if there are certain linear relations between bitangents. In \cite[Theorem 6.1.9]{dolgachev2012} it is shown how to construct all bitangents of a smooth quartic by prescribing seven of them. Most importantly, the coefficients of all other bitangents can be computed by solving linear systems in the coefficients of the given bitangents.

Solving all linear equations symbolically, we get one polynomial equation in the given coefficients such that this equation has a solution if and only if $\dim(\im(\theta_1)+\im(\theta_2))=5$.
Using a computer, one can show that there are indeed no solutions corresponding to smooth quartics.

We note that this does not give any geometric insight but it does give a different proof.
\end{rem}

\section{The Gram spectrahedron}
\label{sec:the gram spectrahedron}


\subsection{Rank 5 Gram tensors}
A general $f$ has a rank 5 extreme point. Consider the sos map $\phi\colon\Rx_2^5\to\Sigma$. Let $V\subset\Rx_2^5$ be the (dense) subset of all tuples $(p_1,\dots,p_5)$ that are quadratically independent, i.e. such that the Gram map $\sy{\spn(p_1,\dots,p_5)}\to \Rx_4$ is injective. The fiber of any element in the image of $V$ has dimension at most 5.
We get $\dim \phi(V)=5\cdot 6 -\binom{5}{2}-\dim F$ where $F$ is the dimension of a generic fiber. We see that we need to have $\dim F=5$ and $\dim\phi(V)=15$ in order for the equation to hold.

\begin{rem}
This is in fact true for the \gs\ of any smooth psd ternary quartic.
\end{rem}

\begin{prop}
\label{prop:codim1square}
Let $U\subset A_2$ be a subspace of codimension 1. Then $\codim U^2\in\{0,2\}$ if $\vc(U)=\emptyset$ and $\codim U^2=3$ if $\vc(U)\neq\emptyset$.
\end{prop}
\begin{proof}
Let $\spn(q)=U^\perp$ with $q\in A_2$. After a \cc\ $q$ is either $x^2$, $x^2+y^2$ or $x^2+y^2+z^2$.
For all three cases we check the claim immediately. Note that the rank of $q$ is 1 iff $\vc(U)\neq\emptyset$.
\end{proof}

\begin{rem}
A similar statement holds in any number of variables. If the rank of the quadratic form $q$ is greater or equal to 3, then $\codim U^2=0$. If the rank is 2, we also get $\codim U^2=2$. And the rank being 1 is equivalent to $\vc(U)\neq\emptyset$ and to $\codim U^2=n$.
\end{rem}

\begin{cor}
\label{cor:possiblefaces}
Let $f\in\Sigma$ be smooth. Let $\theta\in\gram(f)$ be a rank 5 Gram tensor. Then the supporting face of $\theta$ has dimension 0 or 2.
\end{cor}
\begin{proof}
Let $U:=\im(\theta)$ and let $F$ be the supporting face of $\theta$. Then 
\[
\dim F=\binom{6}{2}-\dim U^2\in\{0,2\}.
\]
\end{proof}

\begin{rem}
The shape of a 2-dimensional face of rank 5 is usually an oval with a smooth boundary. However, faces may also degenerate and have singularities and line segments on their boundaries as shown in \cref{fig:fermat2}.

The face on the left-hand side has a smooth and irreducible boundary and represents a 'general' face.
The face on the right-hand side is also a 2-dimensional face of rank 5. Hence, interior points have rank 5 and points on the boundary are of rank 4 or lower. Here, we can find three singular points on the boundary which are rank 3 extreme points. The rest of the boundary consists of two 1-dimensional faces of rank 4 and a \sa\ set of dimension 1 containing rank 4 extreme points.

\begin{figure}[h]
\centering
\begin{minipage}{.5\textwidth}
  \centering
  \includegraphics[scale=0.35]{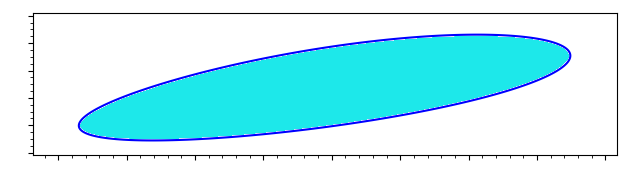}
\end{minipage}%
\begin{minipage}{.5\textwidth}
  \centering
  \includegraphics[scale=0.35]{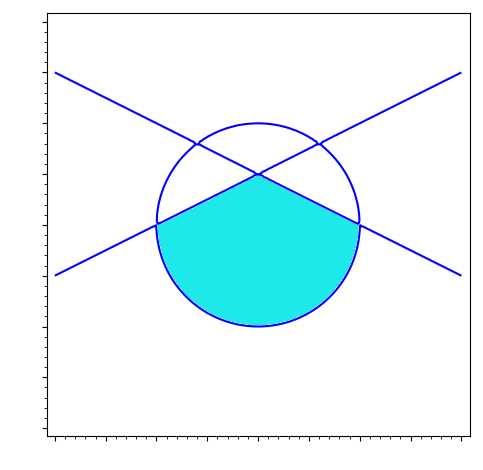}
\end{minipage}
\caption{The face on the left-hand side has a smooth boundary whereas the face on the right-hand side has three singular points and two line segments as part of its boundary.}
\label{fig:fermat2}
\end{figure}

\end{rem}

\subsection{Rank 4 Gram tensors}

We show that for a generic quartic there are no faces of rank $4$ and dimension $1$. Those faces only appear on \gsa\ of quartics that are invariant under some automorphism of $\C^3$ of order $2$. 
Moreover, the dimension of the set of all rank $4$ extreme points is $3$ for a generic quartic.

We start by considering faces of rank 4 and dimension 1. 
Since the boundary of any rank 4 face consists of tensors of rank at most 3, it follows that if the quartic is smooth, the rank 4 faces can only occur as line segments between the eight rank 3 tensors. Since we know the Steiner graph, there are at most 12 1-dimensional faces on the \gs.

\begin{dfn}
Let $f\in A_4$ and $C=\vc(f)\subset\P^2$. We denote by $\Aut(f)$ (or $\Aut(C)$) the \todfn{automorphism group} of $f$ (or $C$) embedded in $\pgl_3(\C)$.
\end{dfn}

\begin{thm}
\label{thm:1dimface}
Let $f\in\Sigma$ be a smooth quartic, then the following are equivalent.
\begin{enumerate}
\item The automorphism group of $f$ has even order.
\item There exist four different bitangents of $f$ that intersect in a common point.
\item After a linear change of coordinates $f=(z^2-f_2)^2-4xy(ax+by)(cx+dy)$ with $f_2\in\C[x,y]_2$ and $a,b,c,d\in\C$.
\item There are three pairwise syzygetic Steiner complexes $\mS_1,\mS_2,\mS_3$ (and corresponding Gram tensors $\theta_1,\theta_2,\theta_3$) and four different bitangents $l_1,\dots,l_4$ such that 
\begin{enumerate}
\item $\{l_1,l_2\},\,\{l_3,l_4\}\in\mS_1$,
\item $\{l_1,l_3\},\,\{l_2,l_4\}\in\mS_2$,
\item $\{l_1,l_4\},\,\{l_2,l_3\}\in\mS_3$,
\end{enumerate}
and $\dim(\im_\C(\theta_i)+\im_\C(\theta_j))\le 4$ for all $i,\,j=1,2,3$.
\end{enumerate}
\end{thm}
\begin{proof}
(i)$\Rightarrow$(ii),(iii): After a linear change of coordinates, the automorphism operates as $(a,b,c)\mapsto (a,b,-c)$ on $\C^3$ and therefore no odd powers of $z$ appear in $f$. Hence, we can write $f$ as
\[
z^4+z^2f_2(x,y)+f_4(x,y)=(z^2+g_2(x,y))^2+g_4(x,y)
\]
where $f_i,g_i\in\C[x,y,z]_i$ for $i=2,4$. Since $g_4$ is a form in 2 variables, it is a product of four linear forms. By definition, those are bitangents of $f$ and intersect in the point $(0:0:1)$.

(ii)$\Rightarrow$(iii): After scaling $f$, we may assume that $f$ is monic in $z$. As $f$ is smooth, the monomial $z^4$ appears in $f$. Wlog the common intersection point of the bitangents is $(0\colon 0\colon 1)$, then the four bitangents $l_1,\dots,l_4$ lie in $\C[x,y]_1$. As four bitangents can only intersect in a common point if they are syzygetic by \cref{lem:azy_bit_no_intersection}, there exists a Steiner complex containing the pairs $\{l_1,l_2\},\{l_3,l_4\}$. Hence, there exists $q\in A_2$, $\lambda\in\C$ such that
\[
f=q^2+\lambda l_1l_2l_3l_4.
\]
We write $q=z^2+zf_1(x,y)+f_2(x,y)$ for some $f_i\in\C[x,y]_i$ ($i=1,2$) and $f_1=ax+by$ for some $a,b\in\C$. Then, the coordinate change $z\mapsto -\frac{a}{2}x-\frac{b}{2}y+z$ yields the desired form in (iii).

(iii)$\Rightarrow$(i): The form $g:=(z^2-f_2)^2-4xy(ax+by)(cx+dy)$ with $a,b,c,d\in\C$ is invariant under the action of the automorphism $\sigma$ defined by $(x,y,z)\mapsto (x,y,-z)$. Let $\phi$ be the automorphism of $\C^3$ mapping $f$ to $g$. Since $g$ is invariant under $\sigma$, the form $f$ is invariant under $\phi^{-1}\sigma\phi$.

(ii)$\Rightarrow$(iv): Let $l_1,l_2,l_3,l_4$ be the four bitangents intersecting in a common point, and let $\mS_1,\mS_2,\mS_3$ be the Steiner complexes that contain $\{l_1,l_2\},\{l_3,l_4\}$ and\\ $\{l_1,l_3\},\{l_2,l_4\}$ and $\{l_1,l_4\},\{l_2,l_3\}$ respectively. Let $\theta_1,\theta_2,\theta_3$ be the corresponding Gram tensors.
 
We prove the statement for $\theta_1,\theta_2$, the other ones work analogously.

Let $q\in A_2$ be the quadratic form such that $f=q^2+\lambda l_1l_2l_3l_4$ for some $\lambda\in\C$. Then, we have $\im_\C(\theta_1)=\spn(l_1l_2,l_3l_4,q)$, $\im_\C(\theta_2)=\spn(l_1l_3,l_2l_4,q)$, and
\[
\im_\C(\theta_1)+\im_\C(\theta_2)=\spn(l_1l_2,l_3l_4,l_1l_3,l_2l_4,q).
\]
By assumption, there are non-zero scalars $\lambda_1,\dots,\lambda_4\in \C$ such that
\[
l_1=\lambda_1 l_2+\lambda_2 l_3\quad \text{ and }\quad l_4=\lambda_3 l_2 +\lambda_4 l_3.
\]
This gives
\begin{align*}
0&=l_1l_4-l_1l_4\\
&=l_1(\lambda_3 l_2 +\lambda_4 l_3)-l_4(\lambda_1 l_2+\lambda_2 l_3)\\
&=-\lambda_1 l_2l_4-\lambda_2 l_3l_4+\lambda_3 l_1l_2+\lambda_4 l_1l_3.
\end{align*}
This shows that the products $l_1l_2,l_3l_4,l_1l_3,l_2l_4$ are linearly dependent and thus 
\[
\dim\spn(l_1l_2,l_3l_4,l_1l_3,l_2l_4,q)\leq 4.
\]

(iv)$\Rightarrow$(ii): Again, consider the Steiner complexes $\mS_1$ and $\mS_2$, the bitangents, and the quadratic form $q$ as above. Then, we know
\[
\im_\C(\theta_1)=\spn(l_1l_2,l_3l_4,q)\quad \text{and}\quad \im_\C(\theta_2)=\spn(l_1l_3,l_2l_4,q).
\]
By assumption, there exist $\lambda_1,\dots,\lambda_5\in\C$ not all zero such that
\[
0=\lambda_1 l_1l_2+\lambda_2 l_3l_4+\lambda_3 l_1l_3+\lambda_4 l_2l_4+\lambda_5 q.
\]
We claim that $\lambda_5=0$. Assume otherwise, then
\[
q=-\frac{1}{\lambda_5}(\lambda_1 l_1l_2+\lambda_2 l_3l_4+\lambda_3 l_1l_3+\lambda_4 l_2l_4).
\]
This gives 
\[
f=(\frac{1}{\lambda_5}(\lambda_1 l_1l_2+\lambda_2 l_3l_4+\lambda_3 l_1l_3+\lambda_4 l_2l_4))^2+\lambda l_1l_2l_3l_4.
\]
But the right-hand side and its derivatives vanish at the intersection point of $l_1$ and $l_4$ which means that this point is a singular point of $f$.

Hence, we get the equation
\[
\lambda_1 l_1l_2+\lambda_3 l_1l_3=-\lambda_2 l_3l_4-\lambda_4 l_2l_4,
\]
or equivalently
\[
l_1(\lambda_1 l_2+\lambda_3 l_3)=l_4(-\lambda_2 l_3-\lambda_4 l_2).
\]
Since any two bitangents are no scalar multiple of each other, there exist $0\neq\nu_1,\nu_2\in\C$ such that
\[
l_1=\nu_1(-\lambda_2 l_3-\lambda_4 l_2)\quad \text{ and }\quad l_4=\nu_2(\lambda_1 l_2+\lambda_3 l_3).
\]
This means that all four bitangents intersect in a common point, namely $\vc(l_2,l_3)$.
\end{proof}

\begin{rem}
\label{rem:auto_subspace}
In the situation of \cref{thm:1dimface}, assume that $f$ has an automorphism $\sigma$ of order 2. Let $l_1,\dots,l_4$ be the four bitangents constructed in (ii). We see that $\sigma$ fixes the 2-dimensional subspace $\spn(l_1,\dots,l_4)\subset A_1$. As $\dim A_1=3$ and $\sigma$ has order 2, this defines $\sigma$.

Therefore, the four bitangents are uniquely defined by $\sigma$. Indeed, \cref{lem:most4bitangents} shows that there is no other bitangent of $f$ contained in $\spn(l_1,\dots,l_4)$ and therefore no other bitangent is fixed by $\sigma$.
\end{rem}

\begin{lem}
\label{lem:most4bitangents}
Let $f\in\Rx_4$ be a smooth quartic. There can be at most four bitangents of $f$ intersecting in a common point.
\end{lem}
\begin{proof}
Assume there are five bitangents $l_1,\dots,l_5$ of $f$ intersecting in a common point. Since no three azygetic bitangents can intersect in a common point by \cref{lem:azy_bit_no_intersection}, the bitangents $l_1,\dots,l_4$ are syzygetic. Let $\mS_1,\mS_2,\mS_3$ be the three syzygetic Steiner complexes containing the bitangents $l_1,\dots,l_4$ in the three different ways. Since any two of them only have these four bitangents in common (\cref{prop:intersection_sc}), the three Steiner complexes contain $3\cdot 12-8=28$ bitangents in total. Hence, one of them contains $l_5$. Assume this is $\mS_1$ and $\{l_1,l_2\},\{l_3,l_4\}\in\mS_1$. Then, the three bitangents $l_1,l_3,l_5$ are azygetic by \cref{lem:azygetic_if_in_sc}, hence cannot intersect in a common point, especially $l_1,\dots,l_5$ do not intersect in a common point.
\end{proof}

\begin{rem}
The \sa\ set of all $f\in\Rx_4$ such that $\Aut(f)$ contains an element of order 2 has codimension 2 in $\Rx_4$.
Especially, a generic form $f\in\Rx_4$ does not have an automorphism of order 2.
\end{rem}

\begin{cor}
\label{cor:one_dim_face_automorphism}
Let $f\in\Sigma$ be a smooth quartic.
\begin{enumerate}
\item If the \gs\ of $f$ has a face of dimension 1, then $f$ has an automorphism of order 2.
\item If $f$ has an automorphism of order 2, there are three rank 3 Gram tensors $\theta_1,\theta_2,\theta_3$ such that $\dim\,(\im_\C(\theta_i)+\im_\C(\theta_j))\le 4$ for all $i,j=1,2,3$. If two of those three rank 3 Gram tensors are real and psd, the \gs\ has a face of dimension 1.
\end{enumerate}
Especially, if $f\in\Sigma$ is chosen generically, its \gs\ has no faces of rank $4$ and dimension $1$.
\end{cor}
\begin{proof}
(i): If $\gram(f)$ has a face $F$ of dimension 1 and rank 4, the boundary of $F$ consists of two Gram tensors $\theta_1,\theta_2$ of rank 3, since $f$ is smooth.

As $\dim\,(\im\theta_1+\im\theta_2)=\dim\im(\theta_1+\theta_2)=4$, it follows from \cref{thm:1dimface} that there exists an automorphism of $f$ of order 2.

(ii) immediately follows from \cref{thm:1dimface}.
\end{proof}

Next, we give an example of a psd ternary quartic with an automorphism of order 2 such that there is no 1-dimensional face on $\gram(f)$.

\begin{rem}
Consider the following family of quartics 
\[
f_{\alpha,\beta}=(z^2+\alpha x^2+\beta y^2)^2+\prod\limits_{j=1}^4 (jx+y)
\]
and $\alpha,\beta\in \R$. Then $f:=f_{\alpha,\beta}$ is psd and smooth if $\alpha,\beta$ are chosen large enough and generic.
Moreover, since $\alpha,\beta$ are chosen generically, this quartic has an automorphism group isomorphic to $C_2$ by \cite[Theorem 6.5.2.]{dolgachev2012}. 

The lines $\vc(jx+y)\,\, (j=1,\dots,4)$ are bitangents of $f$ and intersect in a common point. Let $\mS_1,\mS_2,\mS_3$ be the three syzygetic Steiner complexes that contain these four bitangents as in \cref{thm:1dimface}. Let $\theta_1,\theta_2,\theta_3$ be the corresponding Gram tensors. It follows from \cref{thm:1dimface} that $\dim(\im_\C(\theta_i)+\im_\C(\theta_j))\le 4$ for all $i,j=1,2,3$.

However, since all four of these bitangents are real, it follows from \cite[Proposition 6.6.]{psv2011} that the Gram tensors $\theta_1,\theta_2,\theta_3$ are not real psd.
Especially, this automorphism of order 2 does not give rise to a face of dimension 1 on the \gs.

In fact, there is no 1-dimensional face on $\gram(f)$ at all. Assume there are another four bitangents $l_1,\dots,l_4$ intersecting in a common point. By \cref{lem:most4bitangents} they intersect in a different point. 
From  \cref{rem:auto_subspace} we know that the non-trivial automorphism fixes the 2-dimensional subspace of $\C[x,y,z]_1$ spanned by the four bitangents. This shows that the automorphism cannot fix the first 4 bitangents because then it would fix the whole space. Hence, there are two different automorphisms of order 2, contradicting the fact that $\Aut(f)\cong C_2$.

This shows that although four bitangents are intersecting in a common point, there is no 1-dimensional face on the \gs.
\end{rem}

We have already seen that 1-dimensional faces can only appear as line segments between two rank 3 tensors. However, not all faces represented in the Steiner graph of a smooth quartic can be of rank 4 at the same time.

\begin{prop}
\label{prop:six_one_dim_faces}
Let $f\in\Sigma$ be smooth. There can be at most six 1-dimensional faces on $\gram(f)$.
\end{prop}
\begin{proof}
Let $\mS_1,\,\mS_2$ be two syzygetic Steiner complexes that appear in the Steiner graph. As they are syzygetic, they contain four common bitangents as in \cref{prop:intersection_sc}. If the edge in the Steiner graph connecting the two corresponds to a face of dimension 1 on $\gram(f)$, these four bitangents intersect in a common point.

Looking at the Steiner graph, we see that in every case such a 4-tuple of bitangents contains two bitangents $b_{1j},b_{2i}$ with $\{i,j\}\in\{\{3,4\},\{5,6\},\{7,8\}\}$. Moreover, every of the six edges in one $K_4$ corresponds to a different choice of $\{i,j\}$, namely
\[
b_{13},b_{24},\quad b_{14},b_{23},\quad b_{15},b_{26},\quad b_{16},b_{25},\quad b_{17},b_{28},\quad b_{18},b_{27},
\]
and these are all possibilities.

Assume there are seven or more 1-dimensional faces. Then there exist two 1-dimensional faces corresponding to the 4-tuples of bitangents $b_{1j},b_{2i},l_1,l_2$ and $b_{1j},b_{2i},l_3,l_4$ each intersecting in a common point for some $\{i,j\}\in\{\{3,4\},\{5,6\},$ $\{7,8\}\}$ and some bitangents $l_1,\dots,l_4$.
However, as $b_{1j},b_{2i}$ are contained in both 4-tuples, all bitangents $b_{1j},b_{2i},l_1,l_2,l_3,l_4$ intersect in a common point which is not possible by \cref{lem:most4bitangents}. Note that $l_1,l_2$ and $l_3,l_4$ cannot be the same bitangents as they correspond to different edges in the Steiner graph.
\end{proof}

\begin{rem}
It is not clear if there exists a smooth quartic such that $\gram(f)$ contains six 1-dimensional faces. Even on the \gs\ of the Fermat quartic $x^4+y^4+z^4$, which has a large automorphism group, there are only three 1-dimensional faces.
\end{rem}

This also has implications for rank 5 faces.

\begin{cor}
\label{cor:six_one_dim_faces}
For every smooth $f\in\Sigma$, the \gs\ $\gram(f)$ has a face of rank 5 and dimension 2.
\end{cor}
\begin{proof}
At most 6 faces in the Steiner graph have rank 4 by \cref{prop:six_one_dim_faces}. Therefore, the other faces have rank 5 and dimension 2.
\end{proof}

\begin{prop}
\label{prop:smooth_rk4_tensors}
Let $f\in\Sigma$ be a smooth quartic. Then the \gs\ $\gram(f)$ has an extreme point of rank 4.
\end{prop}
\begin{proof}
By \cref{prop:six_one_dim_faces} there are at most six one-dimensional faces. Let $F$ be any 2-dimensional face in the Steiner graph, say between vertices $\theta_1$ and $\theta_2$. If $F$ is not a polytope, there exists a 1-dimensional \sa\ set of rank 4 extreme points on the boundary of $F$.
Assume $F$ is a polytope, then $F$ is bounded by four 1-dimensional faces. (Note that three are not enough as the line segment connecting $\theta_1$ and $\theta_2$ has to be in the interior of $F$.)

However, now in the second $K_4$ of the Steiner graph there are at most two 1-dimensional faces. Especially any 2-dimensional face in that $K_4$ is not a polytope.
\end{proof}

\begin{rem}
It is true more generally that there are no polyhedral faces of dimension 2 or higher on \gsa\ of (not necessarily smooth) ternary quartics.
\end{rem}

\begin{rem}
\label{rem:ternary_summary}
By generalizing \cite[Corollary 5.4.]{scheiderer2018} we can also calculate the dimension of all rank strata:
We see that for all $f$ in an open, dense, \sa\ subset of $\Sigma$, the \sa\ set of all rank 4 extreme points of $\gram(f)$ has dimension 3, and the set of rank 5 extreme points is dense in the boundary of $\gram(f)$ and has dimension 5. Moreover, the set of all rank 5 points contained in the relative interior of a face of dimension 2 forms a \sa\ set of dimension 4.

Summarizing, we have the following table for $f\in\interior\Sigma$. The first column gives all possible ranks of Gram tensors of $f$. The second one contains all possible dimension of the supporting face of a Gram tensor such rank. The third column gives the dimension of the \sa\ set of all rank $r$ tensors lying in the interior of a face of dimension $s$, and the last one tells us for which ternary quartics such faces appear.\\

\begin{tabular}{c|c|c|c}
rank & dimension of face & dimension of \sa\ set & appearance\\
\hline
3 & 0 & 0 & every smooth $f$\\
4 & 0 & 3 (generic $f$) & every smooth $f$\\
4 & 1 & 1 & only if $2|\#\Aut(f)$\\
5 & 0 & 5 (generic $f$) & every smooth $f$\\
5 & 2 & 4 (generic $f$) & every smooth $f$
\end{tabular}
\end{rem}

\begin{example}
We consider the Fermat quartic $f=x^4+y^4+z^4$. The automorphism group of $f$ has order 96 and can be found in \cite[Theorem 6.5.2.]{dolgachev2012}. We therefore should expect to find several 1-dimensional faces on $\gram(f)$. This is true, and we can even visualize a 3-dimensional slice of $\gram(f)$ where we can see one of the two $K_4$ in the Steiner graph as well as the 2-dimensional faces (see \cref{fig:fermat}).

We consider the 3-dimensional affine slice of $\gram(f)$ spanned by the four rank 3 Gram tensors in one $K_4$ of the Steiner graph of $f$. More precisely, it is the $K_4$ containing the Gram tensor corresponding to the sos representation $f=(x^2)^2+(y^2)^2+(z^2)^2$.
In \cref{fig:fermat} we see the algebraic boundary of this affine slice.
It is given by the determinant of the matrix pencil 
\[
G(\lambda_1,\lambda_2,\lambda_3)=\theta_0+\lambda_1(\theta_1-\theta_0)+\lambda_2(\theta_2-\theta_0)+\lambda_3(\theta_3-\theta_0)
\]
where
{\small
\begin{align*}
\theta_0&=\begin{pmatrix}
1 & 0 & 0 & 0 & 0 & 0\\
0 & 1 & 0 & 0 & 0 & 0\\
0 & 0 & 1 & 0 & 0 & 0\\
0 & 0 & 0 & 0 & 0 & 0\\
0 & 0 & 0 & 0 & 0 & 0\\
0 & 0 & 0 & 0 & 0 & 0
\end{pmatrix},\quad
\theta_1=\begin{pmatrix}
1 & -1 & 0 & 0 & 0 & 0 \\
-1 & 1 & 0 & 0 & 0 & 0 \\
0 & 0 & 1 & 0 & 0 & 0 \\
0 & 0 & 0 & 2 & 0 & 0 \\
0 & 0 & 0 & 0 & 0 & 0 \\
0 & 0 & 0 & 0 & 0 & 0
\end{pmatrix},
\end{align*}
\begin{align*}
\theta_2&=\begin{pmatrix}
1 & 0 & -1 & 0 & 0 & 0 \\
0 & 1 & 0 & 0 & 0 & 0 \\
-1 & 0 & 1 & 0 & 0 & 0 \\
0 & 0 & 0 & 0 & 0 & 0 \\
0 & 0 & 0 & 0 & 2 & 0 \\
0 & 0 & 0 & 0 & 0 & 0
\end{pmatrix},\quad
\theta_3=\begin{pmatrix}
1 & 0 & 0 & 0 & 0 & 0 \\
0 & 1 & -1 & 0 & 0 & 0 \\
0 & -1 & 1 & 0 & 0 & 0 \\
0 & 0 & 0 & 0 & 0 & 0 \\
0 & 0 & 0 & 0 & 0 & 0 \\
0 & 0 & 0 & 0 & 0 & 2
\end{pmatrix}.
\end{align*}}

The part that is contained in $\gram(f)$ is the inner convex part of the red orthant, or more precisely, the (closure of the) connected component of the complement of the algebraic boundary containing the tensor $\frac{1}{4}\sum_{i=0}^3 \theta_i$. 
The Gram tensor $\theta_0$ corresponding to the representation $f=(x^2)^2+(y^2)^2+(z^2)^2$ is the intersection point of the three hyperplanes and lies in the middle of this figure. The other three rank 3 psd tensors in the same $K_4$ in the Steiner graph are $\theta_1,\theta_2,\theta_3$. All of them are connected to $\theta_0$ via a rank 4 and dimension 1 face which is the line segment contained in the intersection of two of the hyperplanes. All other line segments connecting two rank 3 tensors are contained in a 2-dimensional face, each contained in one of the three hyperplanes.

Contrary to what we might expect at first, due to the size of the automorphism group, the other $K_4$ in the Steiner graph of $f$ does not contain any 1-dimensional faces at all.

\begin{figure}[h]
\centering
\begin{tikzpicture}
\node[inner sep=0pt] (3dim) at (0,0)
    {\includegraphics[scale=0.25]{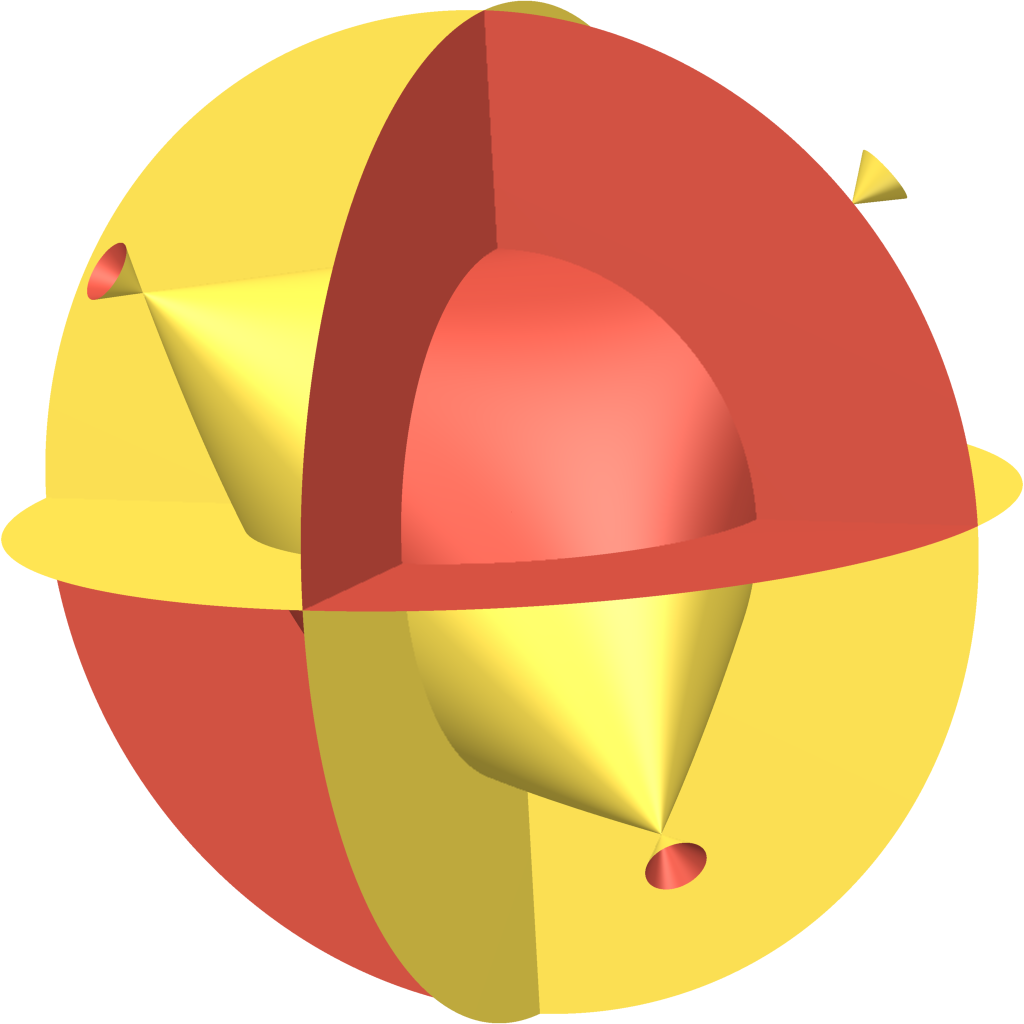}};
\node[circle,fill=black,inner sep=0pt,minimum size=3pt,label=below:{$\theta_2$}] (a) at (2.56*0.84,-0.08*0.84) {};
\node[circle,fill=black,inner sep=0pt,minimum size=3pt,label=above right:{$\theta_3$}] (a) at (-0.15*0.83,2.8*0.83) {};
\node[circle,fill=black,inner sep=0pt,minimum size=3pt,label=below:{$\theta_0$}] (a) at (0,0) {};
 \node[circle,fill=black,inner sep=0pt,minimum size=3pt,label=below:{$\theta_1$}] (a) at (-1.17*0.82,-0.55*0.82) {};
 \draw[dashed] (-1.17*0.82,-0.55*0.82) -- (0,0);
 \draw[dashed] (2.56*0.84,-0.08*0.84) -- (0,0);
 \draw[dashed] (-0.15*0.83,2.8*0.83) -- (0,0);
\end{tikzpicture}
\caption{The algebraic boundary of the \gs\ of the Fermat quartic intersected with the 3-dimensional affine space spanned by the four rank 3 psd Gram tensors $\theta_0,\dots,\theta_3$.}
\label{fig:fermat}
\end{figure}
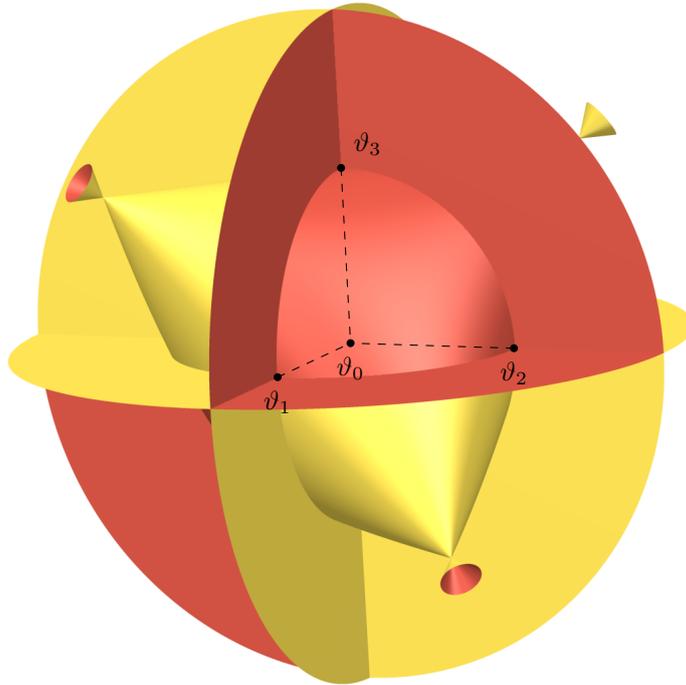
\end{example}

\bibliographystyle{plain}
\bibliography{mybib}
\end{document}